\documentclass[11pt]{article}
\usepackage{latexsym,amsfonts,amsthm,amsmath,amscd,amssymb,color,url}
\usepackage[dvips]{graphicx}

\newtheorem{lemma}{Lemma}[section]
\newtheorem{thm}[lemma]{Theorem}
\newtheorem{rem}[lemma]{Remark}
\newtheorem{prop}[lemma]{Proposition}

\newcommand\matZ{{\mathbb{Z}}}
\newcommand\matC{{\mathbb{C}}}

\newcommand\Sigmatil{{\widetilde\Sigma}}
\newcommand\gtil{{\widetilde{g}}}

\renewcommand{\hbar}{{\overline{h}}}

\newfont{\Got}{eufm10 scaled 1200}
\newcommand{\permu}{{\hbox{\Got S}}}

\newcommand{\compo}{\,{\scriptstyle\circ}\,}

\newcommand{\mycap} [1] {\caption{\footnotesize{#1}}}

\newcommand{\primo}{\textrm{I}}
\newcommand{\secon}{\textrm{I\!I}}
\newcommand{\terzo}{\textrm{I\!I\!I}}
\newcommand{\quart}{\textrm{I\!V}}
\newcommand{\quint}{\textrm{V}}
\newcommand{\sesto}{\textrm{V\!I}}
\newcommand{\setti}{\textrm{V\!I\!I}}

\newcommand{\sistema}[1]{\left\{\begin{array}{l} #1 \end{array}\right.}

\begin{document}

\title{Realizations of certain\\ odd-degree surface branch data}

\author{Carlo~\textsc{Petronio}\thanks{Partially supported by INdAM through GNSAGA, by
MIUR through the PRIN project
n.~2017JZ2SW5$\_$005 ``Real and Complex Manifolds: Topology, Geometry and Holomorphic Dynamics''
and by UniPI through the PRA$\_2018\_22$ ``Geometria e Topologia delle Variet\`a''}}

\maketitle

\begin{abstract}\noindent
We consider surface branch data with base surface the sphere,
odd degree $d$, three branching points, and partitions of $d$ of the form
$$(2,\ldots,2,1)\quad (2,\ldots,2,2h+1)\quad \pi$$
with $\pi$ having length $\ell$. This datum satisfies the Riemann-Hurwitz
necessary condition for realizability if $h-\ell$ is odd and at least $-1$.
For several small values of $h$ and $\ell$ (namely, for $h+\ell\leqslant5$)
we explicitly compute the number $\nu$ of realizations of the datum
up to the equivalence relation
given by the action of automorphisms (even unoriented ones) of both the base and the
covering surface. The expression of $\nu$ depends on arithmetic properties of the
entries of $\pi$. In particular we find that in the only case where
$\nu$ is $0$ the entries of $\pi$ have a common divisor,
in agreement with a conjecture of Edmonds-Kulkarny-Stong
and a stronger one of Zieve.

\smallskip

\end{abstract}

\noindent
In this introduction we first review the notion of
surface branched cover and branch datum, and we define
the \emph{weak Hurwitz number} of a branch datum
(\emph{i.e.}, the number of its realizations up
to a certain ``weak equivalence'' relation).
We then state the new results established in the
rest of the paper, concerning the exact computation
of this number for branch data of a specific type, and we comment on
the connections of these results with an old conjecture of
Edmonds-Kulkarny-Stong and a recent stronger one of Zieve.

\paragraph{Surface branched covers} A surface branched cover is a continuous function
$f:\Sigmatil\to\Sigma$
where $\Sigmatil$ and $\Sigma$ are closed, orientable and connected surfaces and $f$ is locally modeled
on maps of the form
$$(\matC,0)\ni z\mapsto z^m\in(\matC,0).$$
If $m>1$ the point
$0$ in the target $\matC$ is called a \emph{branching point},
and $m$ is called the local degree at the point $0$ in the source $\matC$.
There are finitely many branching points, removing which, together
with their pre-images, one gets a genuine cover of some degree $d$.
If there are $n$ branching points, the local degrees at the points
in the pre-image of the $j$-th one form a partition $\pi_j$ of $d$ of some
length $\ell_j$, and the following Riemann-Hurwitz relation holds:
$$\chi\left(\Sigmatil\right)-(\ell_1+\ldots+\ell_n)=d\left(\chi\left(\Sigma\right)-n\right).$$
Let us now call \emph{branch datum} an array of the form
$$\left(\Sigmatil,\Sigma,d,n,\pi_1,\ldots,\pi_n\right)$$
with $\Sigmatil$ and $\Sigma$ orientable surfaces, $d$ and $n$ positive integers, and $\pi_j$ a partition
of $d$ for $j=1,\ldots,n$.
We say that a branch datum is \emph{compatible} if it satisfies the Riemann-Hurwitz relation.
(Note that $\Sigmatil$ and $\Sigma$ are orientable by assumption; see~\cite{EKS} for
a definition of compatibility in a non-orientable context.)

\paragraph{The Hurwitz problem}
The very old \emph{Hurwitz problem} asks which compatible branch data are
\emph{realizable} (namely, associated to some existing surface branched cover)
and which are \emph{exceptional} (non-realizable).
Several partial solutions
to this problem have been obtained over the time, and we quickly
mention here the fundamental~\cite{EKS}, the survey~\cite{Bologna}, and
the more recent~\cite{Pako, PaPe, PaPebis, CoPeZa, SongXu}.
In particular, for an orientable $\Sigma$ the problem has been shown
to have a positive solution whenever $\Sigma$ has positive genus.
When $\Sigma$ is the sphere $S$, many realizability and exceptionality
results have been obtained (some of experimental nature), but the general
pattern of what data are realizable remains elusive. One guiding
conjecture~\cite{EKS} in this context is that \emph{a compatible branch datum is always
realizable if its degree is a prime number}. It was actually shown in~\cite{EKS}
that proving this conjecture in the special case of $3$ branching
points would imply the general case. This is why many efforts have
been devoted in recent years to investigating the realizability
of compatible branch data with base surface $\Sigma$ the sphere $S$ and having $n=3$
branching points. See in particular~\cite{PaPe, PaPebis} for some evidence
supporting the conjecture.

\paragraph{Hurwitz numbers}
Two branched covers
$$f_1:\Sigmatil\to\Sigma\qquad f_2:\Sigmatil\to\Sigma$$
are said to be \emph{weakly equivalent} if there exist homeomorphisms $\gtil:\Sigmatil\to\Sigmatil$
and $g:\Sigma\to\Sigma$
such that $f_1\compo\gtil=g\compo f_2$, and \emph{strongly equivalent} if
the set of branching points in $\Sigma$ is fixed once and forever and
one can take $g=\textrm{id}_\Sigma$.
The \emph{(weak or strong) Hurwitz number} of a compatible
branch datum is the number of (weak or strong) equivalence classes of branched covers
realizing it. So the Hurwitz problem can be rephrased as the question whether a
Hurwitz number is positive or not (a weak Hurwitz number can be smaller
than the corresponding strong one, but they can only vanish simultaneously).
Long ago Mednykh in~\cite{Medn1, Medn2} gave some formulae for the computation
of the strong Hurwitz numbers,
but the actual implementation of these formulae is rather elaborate in general.
Several results were also obtained in more recent years in~\cite{GKL, KM1, KM2, KML, MSS}.
Some remarks on the different ways of counting the realizations of a branch datum are also contained in~\cite{PeSa}.

\paragraph{Computations}
In this paper we consider branch data of the form
$$\leqno{(\heartsuit)}\qquad\qquad
\left(\Sigmatil,S,2k+1,3,
[2,\ldots,2,1],[2,\ldots,2,2h+1],\pi=\left[d_i\right]_{i=1}^\ell\right)$$
for $h\geqslant0$.  Here we employ square brackets to denote an unordered array
of integers with repetitions.
A direct calculation shows that such a datum is compatible for
$h-\ell=2g-1$, where $g$ is the genus of $\Sigmatil$. So
$h-\ell$ should be odd and at least $-1$, and $g=\frac12(h-\ell+1)$.
We compute the weak Hurwitz number of the datum for $h+\ell\leqslant5$,
namely for the following values of
$(g,h,\ell)$:
$$(0,0,1)\quad (0,1,2)\quad (1,2,1)\quad (0,2,3)\quad (1,3,2)\quad (2,4,1).$$
Organizing the statements according to $g$ and
denoting by $T$ the torus and by $2T$ the genus-2 surface, these are
the results we prove in this article:

\begin{thm}\label{genus0:thm}
\begin{itemize}
\item $(g=0,\ h=0,\ \ell=1)$\quad The branch datum
$$(S,S,2k+1,3,[2,\ldots,2,1],[2,\ldots,2,1],[2k+1])$$
always has a unique realization up to weak equivalence.
\item $(g=0,\ h=1,\ \ell=2)$\quad The branch datum
$$(S,S,2k+1,3,[2,\ldots,2,1],[2,\ldots,2,3],[p,q])$$
always has a unique realization up to weak equivalence.
\item $(g=0,\ h=2,\ \ell=3)$\quad The number $\nu$ of weakly inequivalent realizations of
$$(S,S,2k+1,3,[2,\ldots,2,1],[2,\ldots,2,5],[p,q,r])$$
is as follows:
\begin{itemize}
\item $\nu=0$ if $p=q=r$;
\item $\nu=1$ if two of $p,q,r$ are equal to each other but not all three are;
\item $\nu=2$ if $p,q,r$ are all different from each other and one of them is greater than $k$;
\item $\nu=3$ if $p,q,r$ are all different from each other and all less than or equal to $k$.
\end{itemize}
\end{itemize}
\end{thm}

\begin{thm}\label{genus1:thm}
\begin{itemize}
\item $(g=1,\ h=2,\ \ell=1)$\quad The number of weakly inequivalent realizations of
$$(T,S,2k+1,3,[2,\ldots,2,1],[2,\ldots,2,5],[2k+1])$$
is $\left[\left(\frac k2\right)^2\right]$.
\item $(g=1,\ h=3,\ \ell=2)$\quad The number of weakly inequivalent realizations of
$$(T,S,2k+1,3,[2,\ldots,2,1],[2,\ldots,2,7],[p,q])$$
with $p>q$ is is always positive and given by
\begin{eqnarray*}
& & \left[\left(\frac12\left(k-\left[\frac{p+1}2\right]\right)\right)^2\right]+
\left[\left(\frac12\left[\frac{p-1}2\right]\right)^2\right]+\left[\frac p2\right]^2\\
& - & (p-1)\cdot\left[\frac p2\right]+\left[\left(\frac p2\right)^2\right]
+k^2-k(p-1)+\frac12(p-1)(p-4)
\end{eqnarray*}
except for $k=4$ and $p=7$ where this formula turns the value $6$ but the correct one is $5$.
\end{itemize}
\end{thm}

\begin{thm}\label{genus2:thm}
$(g=2,\ h=4,\ \ell=1)$\quad
The number of weakly inequivalent realizations of
$$(2T,S,2k,3,[2,\ldots,2,1],[2,\ldots,2,9],[2k+1])$$
is $10$ for $k=4$ and otherwise positive and given by
$$\frac k{16}(7k^3 - 42k^2 + 72k - 37)
+\frac58(2k-3)\left[\frac k2\right].$$
\end{thm}

\paragraph{The prime-degree conjecture}
As already mentioned, it was conjectured in~\cite{EKS} that any compatible branch
datum with prime degree is actually realizable, and it was shown
in the same paper that establishing the conjecture with $n=3$
branching points would suffice to prove the general case.
More recently, Zieve~\cite{Zieve} conjectured that an arbitrary compatible branch datum
$$\left(\Sigmatil,\Sigma,d,n,\pi_1,\ldots,\pi_n\right)$$
is realizable provided that
\begin{itemize}
\item $\textrm{GCD}(\pi_j)=1$ for $j=1,\ldots,n$ and
\item $\sum\limits_{j=1}^n\left(1-\frac1{\textrm{lcm}(\pi_j)}\right)\neq 2$.
\end{itemize}
As one easily sees, the compatible branch data with $\sum\limits_{j=1}^n\left(1-\frac1{\textrm{lcm}(\pi_j)}\right)=2$
are precisely those whose associated candidate orbifold cover (see~\cite{PaPe}) is of Euclidean type.
These branch data were fully analyzed in~\cite{PaPe}, where it was shown that indeed some are exceptional
(even with $\textrm{GCD}(\pi_j)=1$ for $j=1,\ldots,n$ in some cases).
So an equivalent way of expressing Zieve's conjecture is to say that a branch datum is
realizable if $\textrm{GCD}(\pi_j)=1$ for $j=1,\ldots,n$ and the datum is not one
of the exceptional ones found in~\cite{PaPe}.
This would imply the prime-degree conjecture, because:
\begin{itemize}
\item If one of the $\pi_i$ reduces to $[d]$ only then the branch datum is
realizable by~\cite{EKS};
\item All the exceptional data of~\cite{PaPe} occur when the degree is composite.
\end{itemize}

We can now remark that our results are in agreement with Zieve's conjecture,
because the only  branch datum for which we compute the weak number Hurwitz number to be $0$
comes from the first case in the last item of Theorem~\ref{genus0:thm}, namely for
a branch datum of the form
$$(S,S,3p,3,[2,\ldots,2,1],[2,\ldots,2,5],[p,p,p])$$
for odd $p\geqslant 3$, and $d=3p$ is composite in this case.

\section{Weak Hurwitz numbers and dessins d'enfant}\label{DA:sec}
In the previous papers~\cite{x1, x3} we have carried out the computation
of weak Hurwitz numbers for different (even-degree) branch data, but the
machine we will employ here is the same used in~\cite{x1, x3}.
We quickly recall it to make the present paper self-contained
(but we omit the proofs). Our techniques are based on the notion of
dessin d'enfant, popularized by Grothendieck in~\cite{Groth} (see also~\cite{Cohen}),
but actually known before his work and
already exploited to give partial answers to the Hurwitz problem (see~\cite{LZ, Bologna}
and the references quoted therein).
Here we explain how to use dessins d'enfant to compute weak Hurwitz numbers.
Let us fix until further notice a branch datum
$$\leqno{(\spadesuit)}\qquad\qquad
\left(\Sigmatil,S,d,3,
\pi_1=\left[d_{1i}\right]_{i=1}^{\ell_1},
\pi_2=\left[d_{2i}\right]_{i=1}^{\ell_2},
\pi_3=\left[d_{3i}\right]_{i=1}^{\ell_3}\right).$$
A graph $\Gamma$ is \emph{bipartite} if it has black and white
vertices, and each edge joins black to white. If $\Gamma$ is
embedded in $\Sigmatil$ we call \emph{region} a component $R$ of
$\Sigmatil\setminus\Gamma$, and
\emph{length} of $R$ the number of white (or black) vertices of
$\Gamma$ to which $R$ is incident (with multiplicity).
A pair $(\Gamma,\sigma)$ is called \emph{dessin d'enfant} representing $(\spadesuit)$
if $\sigma\in\permu_3$ and $\Gamma\subset\Sigmatil$ is a bipartite graph
such that:
\begin{itemize}
\item The black vertices of $\Gamma$ have valence $\pi_{\sigma(1)}$;
\item The white vertices of $\Gamma$ have valence $\pi_{\sigma(2)}$;
\item The regions of $\Gamma$ are discs with lengths $\pi_{\sigma(3)}$.
\end{itemize}
We will also say that $\Gamma$ \emph{represents $(\spadesuit)$ through} $\sigma$.

\begin{rem}
\emph{Let $f:\Sigmatil\to S$ be a branched cover matching $(\spadesuit)$
and take $\sigma\in\permu_3$. If
$\alpha$ is a segment in $S$ with a black and a white end
at the branching points corresponding to $\pi_{\sigma(1)}$ and $\pi_{\sigma(2)}$, then
$\left(f^{-1}(\alpha),\sigma\right)$  represents
$(\spadesuit)$, with vertex colours of $f^{-1}(\alpha)$ lifted via $f$.}
\end{rem}

Reversing the construction described in the previous remark one gets the following:

\begin{prop}\label{from:Gamma:to:f:prop}
To a dessin d'enfant $(\Gamma,\sigma)$ representing $(\spadesuit)$
one can associate a branched cover $f:\Sigmatil\to S$
realizing $(\spadesuit)$, well-defined up to equivalence.
\end{prop}

We next define an equivalence relation $\sim$ on dessins d'enfant as that generated by:
\begin{itemize}
\item $(\Gamma_1,\sigma_1)\sim(\Gamma_2,\sigma_2)$ if $\sigma_1=\sigma_2$ and
there is an automorphism $\gtil:\Sigmatil\to\Sigmatil$ such that
$\Gamma_1=\gtil\left(\Gamma_2\right)$ matching colours;
\item $(\Gamma_1,\sigma_1)\sim(\Gamma_2,\sigma_2)$ if $\sigma_1=\sigma_2\compo(1\,2)$ and
$\Gamma_1=\Gamma_2$ as a set but with vertex colours switched;
\item $(\Gamma_1,\sigma_1)\sim(\Gamma_2,\sigma_2)$ if $\sigma_1=\sigma_2\compo(2\,3)$ and
$\Gamma_1$ has the same black vertices as $\Gamma_2$ and
for each region $R$ of $\Gamma_2$ we have that $R\cap\Gamma_1$ consists
of one white vertex and disjoint edges joining this vertex with the black vertices
on the boundary of $R$.
\end{itemize}

\begin{thm}\label{equiv:Gamma:for:equiv:f:thm}
The branched covers associated as in Proposition~\ref{from:Gamma:to:f:prop}
to two dessins d'enfant are equivalent if and only if
the dessins are related by $\sim$.
\end{thm}

When the partitions $\pi_1,\pi_2,\pi_3$ in the branch datum $(\spadesuit)$
are pairwise distinct, to compute the corresponding weak Hurwitz number one can
stick to dessins d'enfant representing the datum through the identity, namely
one can list up to automorphisms of $\Sigmatil$ the bipartite graphs
with black and white vertices of valence $\pi_1$ and $\pi_2$ and
regions of length $\pi_3$. When the partitions are not distinct, however,
it is essential to take into account the other moves generating $\sim$.
In any case we will henceforth omit any reference to the permutations in $\permu_3$.

\paragraph{Relevant data and repeated partitions}
We now specialize again to a branch datum of the form $(\heartsuit)$.
We will compute its weak Hurwitz number $\nu$ by enumerating up to
automorphisms of $\Sigmatil$ the dessins d'enfant $\Gamma$
representing it through the identity, namely
the bipartite graphs $\Gamma$ with black vertices of
valence $[2,\ldots,2,1]$, the white vertices of valence
$[2,\ldots,2,2h+1]$, and the regions of
length $\pi$.  Ignoring the embedding in $\Sigmatil$, such a $\Gamma$
is abstractly always as shown in Fig.~\ref{abstractgraph:fig},
\begin{figure}
    \begin{center}
    \includegraphics[scale=.9]{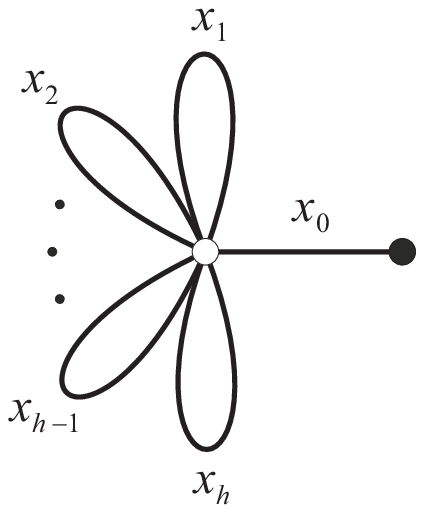}
    \end{center}
\mycap{The abstract dessin d'enfant $\Gamma$. \label{abstractgraph:fig}}
\end{figure}
where $x_0$ stands for $x_0$ alternating black and white $2$-valent vertices,
while $x_i$ stands for $x_i+1$ black and $x_i$ white alternating $2$-valent vertices
for $i>0$. Counting the white vertices we get
$$k-h+1=1+\sum_{i=0}^h x_i\ \Rightarrow\ \sum_{i=0}^h x_i=k-h$$
with of course $x_i\geqslant0$ for all $i$, and no other restriction.
Enumerating these $\Gamma$'s up to automorphisms of $\Sigmatil$
already gives the right value of $\nu$ except if two of the
partitions of $d$ in coincide, and we have:

\begin{prop}\label{repetition:prop}
In a branch datum of the form $(\heartsuit)$ with $h+\ell\leqslant 5$
two of the partitions of $d$ coincide precisely in
the following cases:
\begin{itemize}
\item $(S,S,2k+1,3,[2,\ldots,2,1],[2,\ldots,2,1],[2k+1])$;  %    h=0, ell=1, g=0    thm 0.1(1) gives 1
\item $(S,S,5,3,[2,2,1],[2,3],[2,3])$;                      %    h=1, ell=2, g=0    thm 0.1(2) gives 1
\item $(S,S,9,3,[2,2,2,2,1],[2,2,5],[2,2,5])$;              %    h=2, ell=3, g=0    thm 0.1(3) gives 1
\item $(T,S,5,3,[2,2,1],[5],[5])$;                          %    h=2, ell=1, g=1    thm 0.2(1) gives 1 (k=2)
\item $(T,S,9,3,[2,2,2,2,1],[2,7],[2,7])$;                  %    h=3, ell=2, g=1    Exceptional case
\item $(2T,S,9,3,[2,2,2,2,1],[9],[9])$.                     %    h=4, ell=1, g=2    Exceptional case
\end{itemize}
\end{prop}

\begin{proof}
The lengths of the partitions $\pi_1,\pi_2,\pi$ in $(\heartsuit)$ are $\ell_1=k+1$, $\ell_2=k-h+1$ and
$\ell=h+1-2g$.

We have $\pi_1=\pi_2$ only for $h=0$, $\ell=1$ and $g=0$, whence
the first listed item.

We have $\pi_1=\pi$ only for $k+1=h+1-2g$, whence $h-k=2g\geqslant 0$, but of course $h\leqslant k$, so
$h=k$ and the first listed item again.

We have $\pi_2=\pi$ only for $k-h+1=\ell$, so $k=h+\ell-1$, whence in particular $k\leqslant 4$,
and listing the relevant cases is straightforward.
\end{proof}

While proving our results, for the first four data of the previous statement
we will find that there is
a unique $\Gamma$ up to automorphisms of $\Sigmatil$ giving a realization.
In these cases, we will not need to consider the second and
third generating moves of $\sim$, but for the last two data we will have to
do this, actually getting a correction to the computation.

\section{Genus 0}\label{genus0:sec}
In this section we prove Theorem~\ref{genus0:thm}.

\medskip

For $h=0$ and $\ell=1$ the graph $\Gamma$ of Fig.~\ref{abstractgraph:fig}
reduces to a segment, so of course it has a unique embedding in $S$ and
the conclusion is obvious.

\bigskip

For $h=1$ and $\ell=2$ the embedding is again
unique, and it realizes $[2x_0+x_1+2,x_1+1]$.
Assuming $p>q$, namely $k+1\leqslant p\leqslant 2k$ and $q=2k+1-p$,
we get the unique realization of the
datum choosing $x_0=p-k-1$ and $x_1=2k-p$.

\bigskip

Turning to the case $h=2$ and $\ell=3$, we now have two embeddings
of $\Gamma$ in $S$, shown in Fig.~\ref{h2ell3:fig}
\begin{figure}
    \begin{center}
    \includegraphics[scale=.6]{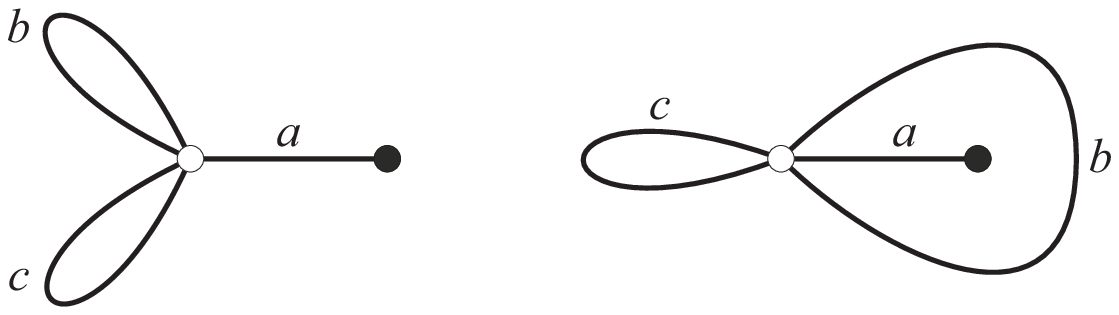}
    \end{center}
\mycap{Embeddings of $\Gamma$ in $S$ for $h=2$ and $\ell=3$. \label{h2ell3:fig}}
\end{figure}
and denoted by $\primo(a,b,c)$ and $\secon(a,b,c)$
---for the sake of simplicity we use from now on letters such as $a,b,c$ instead of $x_0,\ldots,x_h$.
These graphs realize $[2a+b+c+3,b+1,c+1]$ and $[2a+b+2,b+c+2,c+1]$ respectively.
Moreover $\primo(a,b,c)$ has a symmetry switching $b$ and $c$, while $\secon(a,b,c)$ has no symmetries.
Let us now assume $p\geqslant q\geqslant r$.

\medskip
\noindent
\textsc{Claim I}: The number of realizations of $[p,q,r]$ through $\primo(a,b,c)$ is $1$ if $p>k$ and $0$ otherwise.

\medskip
\noindent
\textsc{Proof of Claim I}: Since $2a+b+c+3$ is greater than $b+1$ and $c+1$, we can realize $[p,q,r]$ only  with
$$\sistema{p=2a+b+c+3\\ q=b+1\\ r=c+1}$$
(for $q>r$ we might as well choose $q=c+1$ and $r=b+1$, but the $b\leftrightarrow c$ symmetry of $\primo(a,b,c)$
makes this alternative immaterial). Noting that $q+r=2k+1-p$ one sees that the system as unique solution
$$\sistema{a=p-k-1\\ b=q-1\\ c=r-1}$$
which is acceptable precisely for $p>k$.

\medskip
Before proceeding with another claim we note that we can split the possibilities for
$[p,q,r]$ in $6$ mutually exclusive cases $IJ/M$, where
\begin{itemize}
\item $I,J\in\{G,E\}$ with $G$ standing for $>$ and $E$ standing for $=$
\item $M\in\{G,L\}$ with $G$ standing for $>$ and $L$ standing for $\leqslant$
\item $IJ/M=\{[p,q,r]:\ p\,I\,j\,J\,q,\ p\,M\,k\}$
\item $EE/G=EG/G=\emptyset$, so we write $EE$ and $EG$ instead of $EE/L$ and $EG/L$.
\end{itemize}
So Claim I states that there is one realization through $\primo(a,b,c)$
in cases $GG/G$ and $GE/G$ and none in the other cases.

\medskip
\noindent
\textsc{Claim \secon}: The number of realizations of $[p,q,r]$ through $\secon(a,b,c)$ is as follows:
\begin{itemize}
\item[0] in cases $EE$ and $GE/G$;
\item[1] in cases $EG$, $GE/L$ and $GG/G$;
\item[3] in case $GG/L$.
\end{itemize}

\medskip
\noindent
\textsc{Proof of Claim \secon}: Since $b+c+2>c+1$ case $EE$ cannot be realized.
For $p=q>r$ (case $EG$) we can only have
$$\sistema{p=2a+b+2\\ p=b+c+2\\ r=c+1}\ \Leftrightarrow\ \sistema{a=k-p\\ b=p-r-1\\ c=r-1}$$
and the solution is acceptable because $p\leqslant k$.
For $p>q=r$ (case $GE$) we can only have
$$\sistema{p=b+c+2\\ q=2a+b+2\\ q=c+1}\ \Leftrightarrow\ \sistema{a=k-p\\ b=p-q-1\\ c=q-1}$$
which is acceptable precisely for $p\leqslant k$, so there is no realization
in case $GE/G$ and one in case $GE/L$.
For $p>q>r$ (case $GG$) there are three possibilities:
$$\sistema{p=2a+b+2\\ q=b+c+2\\ r=c+1}\ \Leftrightarrow\ \sistema{a=k-q\\ b=q-r-1\\ c=r-1}$$
which is always acceptable, and
$$\sistema{p=b+c+2\\ q=2a+b+2\\ r=c+1}\ \Leftrightarrow\ \sistema{a=k-p\\ b=p-r-1\\ c=r-1}$$
$$\sistema{p=b+c+2\\ q=c+1\\ r=2a+b+2}\ \Leftrightarrow\ \sistema{a=k-p\\ b=p-q-1\\ c=q-1}$$
which are acceptable for $p\leqslant k$, whence 1 realization in case $GG/G$ and $3$ in case $GG/L$.

\medskip
\noindent
\textsc{Conclusion}: The number of realizations of $[p,q,r]$ through
$\primo+\secon$ is $0+0=0$ in case $EE$, $1+0=1$ in case $GE/G$, $0+1=1$
in case $GE/L$, $0+1=1$ in case $EG$, $1+1=2$ in case $GG/G$ and $0+3=3$ in case $GG/L$.

\medskip

Note that for the branch datum
$$(S,S,9,3,[2,2,2,2,1],[2,2,5],[2,2,5])$$
from Proposition~\ref{repetition:prop} we have found
$\nu=1$ already, so we do not have to worry about the repetitions in the partitions.
The proof is complete.

\section{Genus 1}\label{genus1:sec}
In this section we prove Theorem~\ref{genus1:thm}.

\medskip

For $h=2$ and $\ell=1$ the graph $\Gamma$ of Fig.~\ref{abstractgraph:fig}
has a unique embedding in $T$ with a single disc as a complement, as shown
in Fig.~\ref{h2ell1:fig}.
\begin{figure}
    \begin{center}
    \includegraphics[scale=.6]{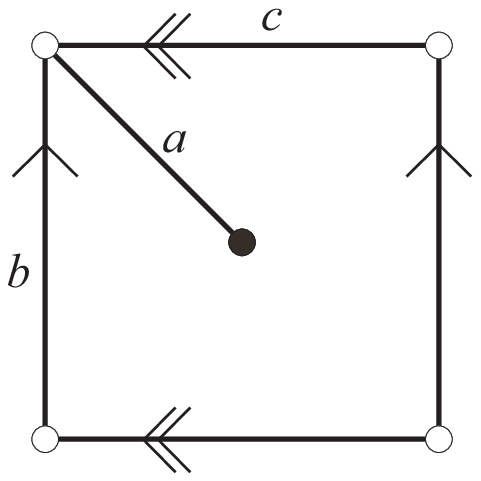}
    \end{center}
\mycap{Embedding of $\Gamma$ in $T$ for $h=2$ and $\ell=1$. \label{h2ell1:fig}}
\end{figure}
This graph is subject to the symmetry $b\leftrightarrow c$, so the number
of realizations of the branch datum equals the number of expressions
$k-2$ as $a+b+c$ with $a,b,c\geqslant0$ up to $b\leftrightarrow c$, namely
$$\sum_{a=0}^{k-2}\left(\left[\frac{k-2-a}2\right]+1\right)=
\sum_{a=0}^{k-2}\left[\frac{k-a}2\right]=
 \sum_{n=2}^{k}\left[\frac n2\right]=
 \sum_{n=0}^{k}\left[\frac n2\right]=\left[\left(\frac k2\right)^2\right].$$

\bigskip

For $h=3$ and $\ell=2$ we first determine the embeddings
in $T$ of the bouquet $B$ of $3$ circles with two discs as regions. Of course at least a circle of $B$
is non-trivial on $T$, so its complement is an annulus. Then another circle must join the boundary components
of this annulus, so we can assume two circles of $B$ form a standard meridian-longitude pair on $T$.
Then the possibilities for $B$ up to automorphisms of $T$ are as in Fig.~\ref{bouquetinTbis:fig}.
\begin{figure}
    \begin{center}
    \includegraphics[scale=.6]{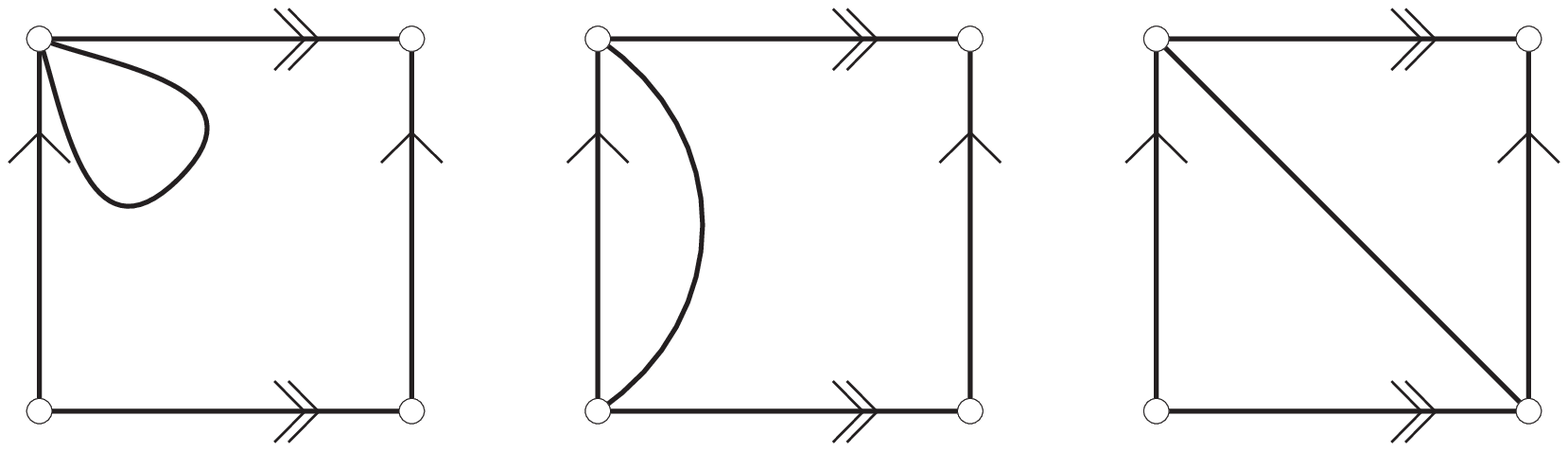}
    \end{center}
\mycap{A bouquet of 3 circles in $T$ with 2 discs as regions. \label{bouquetinTbis:fig}}
\end{figure}
Note that these embeddings have respectively a $\matZ/_2$, a $\matZ/_2\times\matZ/_2$ and a
$\permu_3\times\matZ/_2$, symmetry. It easily follows that the relevant embeddings in $T$ of $\Gamma$ are
up to automorphisms those shown in Fig.~\ref{h3ell2:fig}.
\begin{figure}
    \begin{center}
    \includegraphics[scale=.6]{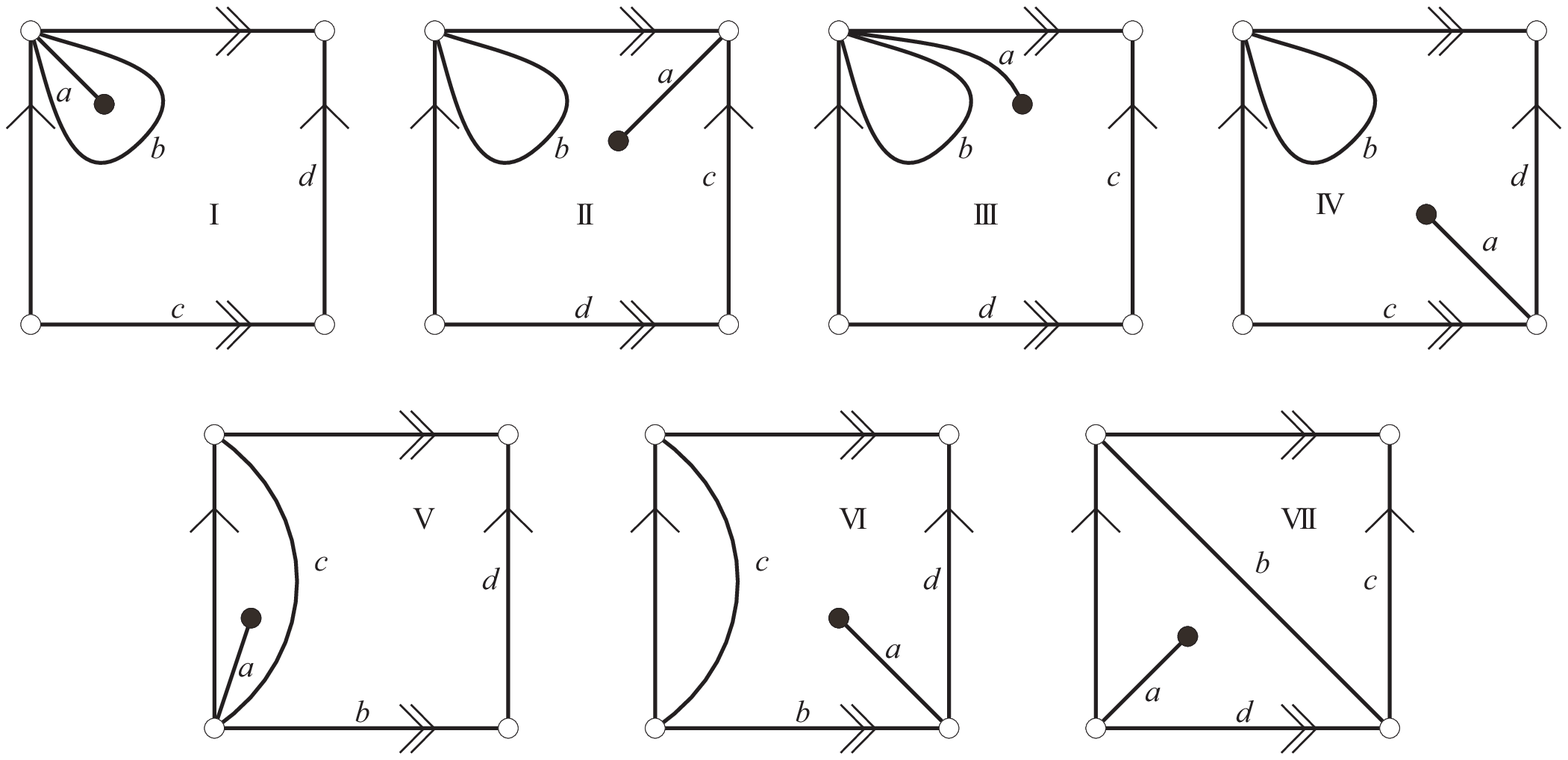}
    \end{center}
\mycap{Embeddings in $T$ of $\Gamma$ with 2 discs as regions. \label{h3ell2:fig}}
\end{figure}
Note that we have a symmetry switching $c$ and $d$ in cases $\primo,\ \quart,\ \quint,\ \setti$, and no other one.
Moreover the different embeddings of $\Gamma$ realize the following $\pi$'s:
\begin{itemize}
\item $\primo(a,b,c,d)\ \longrightarrow\ (2a+b+2,b+2c+2d+5)$
\item $\secon(a,b,c,d)\ \longrightarrow\ (b+1,2a+b+2c+2d+6)$
\item $\terzo(a,b,c,d)\ \longrightarrow\ (b+1,2a+b+2c+2d+6)$
\item $\quart(a,b,c,d)\ \longrightarrow\ (b+1,2a+b+2c+2d+6)$
\item $\quint(a,b,c,d)\ \longrightarrow\ (2a+c+d+3,2b+c+d+4)$
\item $\sesto(a,b,c,d)\ \longrightarrow\ (2a+2b+c+d+5,c+d+2)$
\item $\setti(a,b,c,d)\ \longrightarrow\ (2a+b+c+d+4,b+c+d+3).$
\end{itemize}
We will count the realizations of $\pi=[p,q]$ assuming $p>q$, namely $p>k$ and $q=2k+1-p$,
and analyzing case after case the contribution of each of the graphs \primo\ to $\setti$.
Along the way we will discuss all the cases where the contribution is null,
which will only happen when $p$ is close to its lower bound $k+1$ or upper bound $2k$.
Occasionally, to be completely precise, we would need to discuss separately
some small values of $p$ (and hence $k$), for which the contribution is also null, but
as a matter of fact all these cases are included in the general ones, as the
reader can easily check.

\medskip
\noindent
\textsc{Claim I}: The number of realizations of $[p,q]$ through $\primo(a,b,c,d)$ is
\begin{equation}\label{I}
\left[\left(\frac12\left(k-\left[\frac{p+1}2\right]\right)\right)^2\right]+
\left[\left(\frac12\left[\frac{p-1}2\right]\right)^2\right]-
\left[\left(\frac{p-k-1}2\right)^2\right].\end{equation}

\medskip
\noindent
\textsc{Proof of Claim I}: We first count the non-negative solutions $a,b,c,d$ up to the symmetry $c\leftrightarrow d$ of the system
$$\sistema{2a+b+2=p\\ b+2c+2d+5=2k+1-p.}$$
To begin, we state that $(a,b,c,d)$ solve the system if and only if they satisfy the conditions
$$\sistema{c,d\geqslant0\\ c+d\leqslant k-2-\left[\frac{p+1}2\right]\\
a=p-k+1+c+d\\ b=2k-4-p-2c-2d.}$$
In fact, if $(a,b,c,d)$ solve the system then (from the second equation)
\begin{eqnarray*}
& & c+d=\frac12(2k-p-4-b)=k-2-\frac12(p+b)\\
&\Rightarrow& c+d\leqslant k-2-\frac p2\ \Leftrightarrow\ c+d\leqslant k-2-\left[\frac{p+1}2\right]
\end{eqnarray*}
and the expression of $a,b$ in terms of $p,k,c,d$ is readily derived. Conversely we must show that if
$c,d\geqslant 0$ and $c+d\leqslant k-2-\frac p2$ then the expressions
$$a=p-k+1+c+d\qquad b=2k-4-p-2c-2d$$
turn non-negative values. For $a$, this is true because $p>k$ (so actually $a\geqslant2$) and for $b$
it is true because $c+d\leqslant k-2-\frac p2$. The statement implies that the number of solutions is $0$ for
$k-2-\frac p2<0$, namely for $p>2k-4$, while otherwise it is
\begin{equation}\label{Ii}
\sum_{n=0}^{k-2-\left[\frac{p+1}2\right]}\left(\left[\frac n2\right]+1\right)=
\left[\left(\frac12\left(k-2-\left[\frac{p+1}2\right]\right)\right)^2\right]+k-1-\left[\frac{p+1}2\right]
\end{equation}
but a straight-forward argument shows that the expression on the right-hand side of~(\ref{Ii})
gives the correct value $0$ also for $2k-4<p\leqslant 2k$.
We next count the non-negative solutions $a,b,c,d$ up to the symmetry $c\leftrightarrow d$ of the system
$$\sistema{2a+b+2=2k+1-p\\ b+2c+2d+5=p}$$
and we state that $(a,b,c,d)$ solve the system if and only if they satisfy the conditions
$$\sistema{c,d\geqslant0\\ p-k-2\leqslant c+d\leqslant \left[\frac{p-1}2\right]-2\\
a=k-p+2+c+d\\ b=p-2c-2d-5.}$$
In fact, if $(a,b,c,d)$ solve the system then (from the second equation)
$$c+d=\frac12(p-5-b)\ \Rightarrow\ c+d\leqslant\frac{p-1}2-2\ \Leftrightarrow\ c+d\leqslant\left[\frac{p-1}2\right]-2.$$
Moreover the expressions of $a,b$ in terms of $p,k,c,d$ are readily obtained, and that of $a$ implies that
$c+d\geqslant p-k-2$. Conversely, for $c,d\geqslant0$ and $p-k-2\leqslant c+d\leqslant \frac{p-1}2-2$ we see that
$a=k-p+2+c+d$ and $b=p-2c-2d-5$ are non-negative. Now recall that $p>k$, so $p-k-2<0$ only for $p=k+1$, in which case
the number of solutions is
\begin{equation}\label{Iiicaselimitfirst}
\sum_{n=0}^{\left[\frac k2\right]-2}\left(\left[\frac n2\right]+1\right)=
\left[\left(\frac12\left[\frac k2\right]-1\right)^2\right]+\left[\frac k2\right]-1=
\left[\left(\frac12\left[\frac k2\right]\right)^2\right].
\end{equation}
Moreover we have
$$\left[\frac{p-1}2\right]-2<p-k-2\ \Leftrightarrow\ \frac{p-1}2<p-k\ \Leftrightarrow\ p>2k-1\ \Leftrightarrow p=2k$$
in which case there are no solutions. For $k+1<p<2k$ we have instead
\begin{equation}\label{Iii}
\begin{array}{cl}
 & \displaystyle{\sum\limits_{n=p-k-2}^{\left[\frac{p-1}2\right]-2}\left(\left[\frac n2\right]+1\right)}\\
=& \displaystyle{\left[\left(\frac12\left[\frac{p-1}2\right]-1\right)^2\right]-
\left[\left(\frac{p-k-3}2\right)^2\right]+\left[\frac{p-1}2\right]-p+k+1}
\end{array}
\end{equation}
but the expression on the right-hand side of~(\ref{Iii}) is seen to coincide with~(\ref{Iiicaselimitfirst}) for $p=k+1$ and
to vanish for $p=2k$.
To conclude we must check that the sum of the two expressions on the right-hand sides of~(\ref{Ii}) and~(\ref{Iii})
give the claimed value~(\ref{I}), which only requires a little manipulation that we omit here.

Before turning to the next case, we note that the number of realizations of $(p,q)$ through \primo\ is
always positive except for $p=2k$ (this follows from the proof of formula~(\ref{I}) rather than from its
expression).

\medskip
\noindent
\textsc{Claim $\secon+\terzo$}: The number of realizations of $[p,q]$ through each of $\secon(a,b,c,d)$ and
$\terzo(a,b,c,d)$ is
\begin{equation}\label{II+III}
\frac12(p-k-1)(p-k-2).
\end{equation}

\medskip
\noindent
\textsc{Proof of Claim $\secon+\terzo$}: Since $2a+b+2c+2d+6>b+1$ the only realizations come from the
solutions of
$$\sistema{2a+b+2c+2d+6=p\\ b+1=2k+1-p}$$
(and, as a matter of fact, there are no solutions if $p-(2k+1-p)<5$, namely for $p\leqslant k+2$).
The solutions we seek come with
$$b=2k-p\qquad a+c+d=p-k-3$$
so there are
$$\sum_{a=0}^{p-k-3}(p-k-3-a+1)=\frac12(p-k-1)(p-k-2)$$
of them, and this expression is correct also for $p=k+1$ and $p=k+2$ (which are the only cases
where there are no realizations).

\medskip
\noindent
\textsc{Claim $\quart$}: The number of realizations of $[p,q]$ through $\quart(a,b,c,d)$ is
\begin{equation}\label{IV}
\left[\left(\frac{p-k-1}2\right)^2\right].
\end{equation}

\medskip
\noindent
\textsc{Proof of Claim $\quart$}: The situation is identical to the previous one, except that
now we have the symmetry $c\leftrightarrow d$ to take into account, so the number of realizations is
$$\sum_{a=0}^{p-k-3}\left[\frac{p-k-3-a+2}2\right]=
\sum_{n=2}^{p-k-1}\left[\frac n2\right]=
\sum_{n=0}^{p-k-1}\left[\frac n2\right]=\left[\left(\frac{p-k-1}2\right)^2\right]$$
which again is correct also for $p=k+1$ and $p=k+2$ (only cases
where there are no realizations).

\medskip
\noindent
\textsc{Claim $\quint$}: The number of realizations of $[p,q]$ through $\quint(a,b,c,d)$ is
\begin{equation}\label{V}
\left[\frac p2\right]^2-(p-1)\cdot\left[\frac p2\right]-k(p-k)+\frac12p(p-1).
\end{equation}

\medskip
\noindent
\textsc{Proof of Claim $\quint$}: We first count the non-negative integer solutions $(a,b,c,d)$ up to the
$c\leftrightarrow d$ symmetry of
$$\sistema{2a+c+d+3=p\\ 2b+c+d+4=2k+1-p}$$
noting that there is none if $2k+1-p\leqslant 3$, namely for $p\geqslant 2k-2$, so we assume $p\leqslant 2k-3$.
We first state that $(a,b,c,d)$ is a solution if and only if
$$\sistema{p-k\leqslant a\leqslant \left[\frac{p-1}2\right]-1\\ c+d=p-2a-3\\ b=k-p+a}$$
and $0\leqslant p-k\leqslant \left[\frac{p-1}2\right]-1$. The last assertion is
easy since $p>k$, and $p-k$ equals $\left[\frac{p-1}2\right]-1$ precisely for $p=2k-3$ and $p=2k-4$, while
it is strictly less for smaller $p$. Now if $(a,b,c,d)$ is a solution we have
$$\sistema{
a=\frac{p-3-c-d}2\leqslant\frac{p-3}2\ \Rightarrow\ a\leqslant\left[\frac{p-3}2\right]=\left[\frac{p-1}2\right]-1\\
c+d=p-2a-3\\
b=\frac{2k+1-p-p+2a+3-4}2=k-p+a\ \Rightarrow\ a\geqslant p-k.}$$
The sufficiency of these conditions for $(a,b,c,d)$ to be a solution is proved very similarly.
We then have the count
\begin{eqnarray}
\nonumber & & \sum_{a=p-k}^{\left[\frac{p-1}2\right]-1}\left(\left[\frac{p-2a-3}2\right]+1\right)\\
\nonumber &=& \sum_{a=p-k}^{\left[\frac{p-1}2\right]-1}\left(\left[\frac{p-1}2\right]-a-1+1\right)\\
\nonumber &=& \left[\frac{p-1}2\right]\cdot\left(\left[\frac{p-1}2\right]-1-(p-k)+1\right)\\
\nonumber & & -\frac12\left(\left[\frac{p-1}2\right]-1\right)\cdot\left[\frac{p-1}2\right]
+\frac12(p-k-1)(p-k)\\
&=& \frac12\left[\frac{p-1}2\right]^2-\frac12\left[\frac{p-1}2\right](2p-2k-1)+\frac12(p-k-1)(p-k)
\label{Vi}
\end{eqnarray}
which is readily seen to give the correct value $0$ also for $2k-2\leqslant p\leqslant 2k$. The argument for the system
$$\sistema{2a+c+d+3=2k+1-p\\ 2b+c+d+4=p}$$
is very similar. There are solutions for $p\leqslant 2k-2$ and they correpond to
$$\sistema{p-k-1\leqslant b\leqslant\left[\frac p2\right]-2\\
c+d=p-2b-4\\ a=k-p+b+1}$$
so there are
\begin{eqnarray}
\nonumber & & \sum_{b=p-k-1}^{\left[\frac{p}2\right]-2}\left(\left[\frac{p-2b-4}2\right]+1\right)\\
\nonumber &=& \sum_{b=p-k-1}^{\left[\frac{p}2\right]-2}\left(\left[\frac{p}2\right]-b-2+1\right)\\
\nonumber &=& \left(\left[\frac{p}2\right]-1\right)\cdot\left(\left[\frac{p}2\right]-2-(p-k-1)+1\right)\\
\nonumber & & -\frac12\left(\left[\frac{p}2\right]-2\right)\cdot\left(\left[\frac{p}2\right]-1\right)
+\frac12(p-k-2)(p-k-1)\\
&=&\frac12 \left[\frac{p}2\right]^2-\frac12\left[\frac{p}2\right](2p-2k-1)+\frac12(p-k-1)(p-k)
\label{Vii}
\end{eqnarray}
of them, and the formula is correct also for $2k-1\leqslant p\leqslant 2k$.
To conclude we must now show that summing~(\ref{Vi}) and~(\ref{Vii}) we get~(\ref{V}), which
is proved with a little patience noting that $\left[\frac{p-1}2\right]+\left[\frac p2\right]=p-1$.

\medskip
\noindent
\textsc{Claim $\sesto$}: The number of realizations of $[p,q]$ through $\sesto(a,b,c,d)$ is
\begin{equation}\label{VI}
(2k-p)(p-k-1).
\end{equation}

\medskip
\noindent
\textsc{Proof of Claim $\sesto$}:
We must count the solutions of
$$\sistema{2a+2b+c+d+5=p\\ c+d+2=2k+1-p.}$$
For $p=k+1$ and $p=2k$ there is no solution, otherwise the solutions $(a,b,c,d)$ are the 4-tuples such that
$$c+d=2k-1-p\qquad a+b=p-k-2$$
so there are $(2k-p)(p-k-1)$ of them as claimed, and the formula is correct for $p=k+1$ and $p=2k$ as well.

\medskip
\noindent
\textsc{Claim $\setti$}: The number of realizations of $[p,q]$ through $\setti(a,b,c,d)$ is
\begin{equation}\label{VII}
\left[\left(\frac p2\right)^2\right]-k(p-k).
\end{equation}

\medskip
\noindent
\textsc{Proof of Claim $\setti$}:
We must count the solutions of
$$\sistema{2a+b+c+d+4=p\\ b+c+d+3=2k+1-p}$$
up to $c\leftrightarrow d$, and there is none for $p\geqslant 2k-1$.
Otherewise the system is equivalent to
$$b+c+d=2k-p-2\qquad a=p-k-1$$
so the number of solutions is

\begin{eqnarray*}
  & & \sum_{b=0}^{2k-p-2}\left(\left[\frac{2k-p-2-b}2\right]+1\right) \\
  &=& \sum_{n=0}^{2k-p-2}\left(\left[\frac{n}2\right]+1\right) \\
  &=& \left[\left(\frac{2k-p-2}2\right)^2\right]+2k-p-1 \\
  &=& \left[\left(k-1-\frac p2\right)^2\right]+2k-p-1 \\
  &=& k^2-2k+1-(k-1)p+\left[\left(\frac p2\right)^2\right]+2k-p-1 \\
  &=& \left[\left(\frac p2\right)^2\right]-k(p-k)
\end{eqnarray*}
which turns the right value $0$ also for $p\geqslant 2k-1$.

\medskip

Summing the contributions from \primo\ to \setti\ the
expression in the statement of Theorem~\ref{genus1:thm} is now easily obtained, but
we still have to worry about the penultimate item in
Proposition~\ref{repetition:prop}.  The above discussion or a direct inspection show
that this datum is realized by the graphs
$$\begin{array}{ccc}
\primo(0,0,1,0) & \secon(0,1,0,0) & \terzo(0,1,0,0)\\
\quart(0,1,0,0) & \sesto(1,0,0,0) & \sesto(0,1,0,0)
\end{array}$$
shown in Fig.~\ref{genus1duality:fig}.
\begin{figure}
    \begin{center}
    \includegraphics[scale=.6]{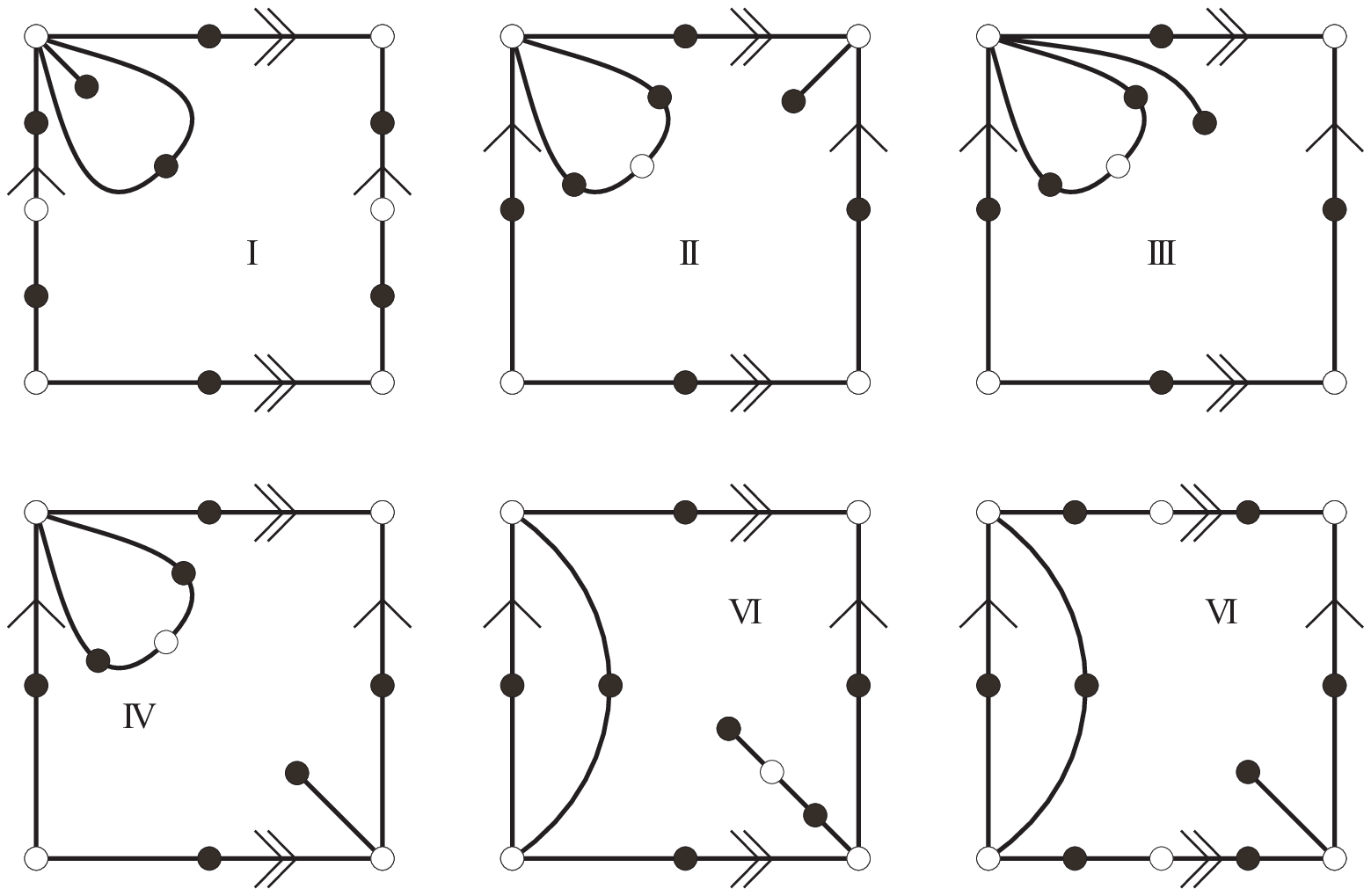}
    \end{center}
\mycap{Graphs in $T$ realizing $(T,S,9,3,[2,2,2,2,1],[2,7],[2,7])$. \label{genus1duality:fig}}
\end{figure}
Since the second and third partition of the datum coincide, all
we have to do is to check whether any of these graphs are dual to each other
under the last transformation generating
the equivalence $\sim$ of Theorem~\ref{equiv:Gamma:for:equiv:f:thm}.
This is done in Figg.~\ref{genus1duality15:fig} to~\ref{genus1duality6:fig}, and
the conclusion is that the number of inequivalent realizations of the
datum is $5$ rather than $6$, as in the statement of Theorem~\ref{genus1:thm}.
\begin{figure}
    \begin{center}
    \includegraphics[scale=.8]{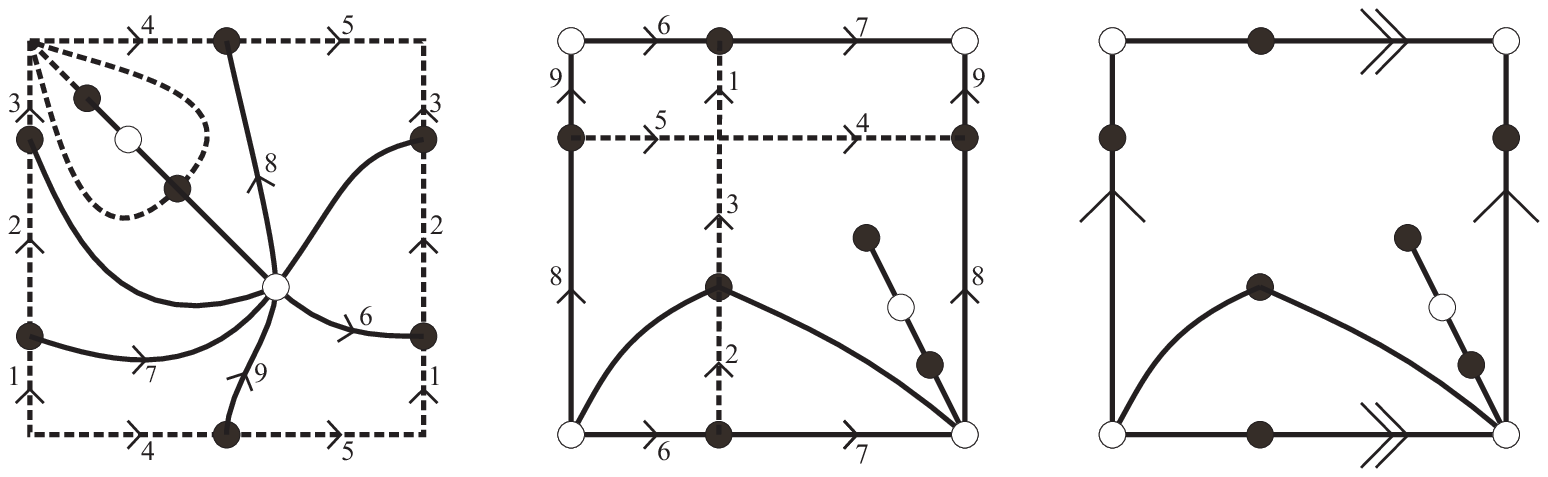}
    \end{center}
\mycap{The dual to $\primo(0,0,1,0)$ is $\sesto(1,0,0,0)$.
On the left the dotted lines are the edges of the original graph and the solid lines
are the edges of the dual graph; some edges of the original graph
are oriented and numbered from $1$ to $5$ to encode the way they should be identified to each other;
some edges of the dual graph are also oriented and numbered from $6$ to $9$;
in the center we show the result of cutting along the edges from $6$ to $9$
and gluing along those from $1$ to $5$; on the right we show the same figure as
in the center but deleting the original graph, from which one easily sees the type of the dual graph.
Similar explanations apply to the next four figures.\label{genus1duality15:fig}}
\end{figure}

\begin{figure}
    \begin{center}
    \includegraphics[scale=.8]{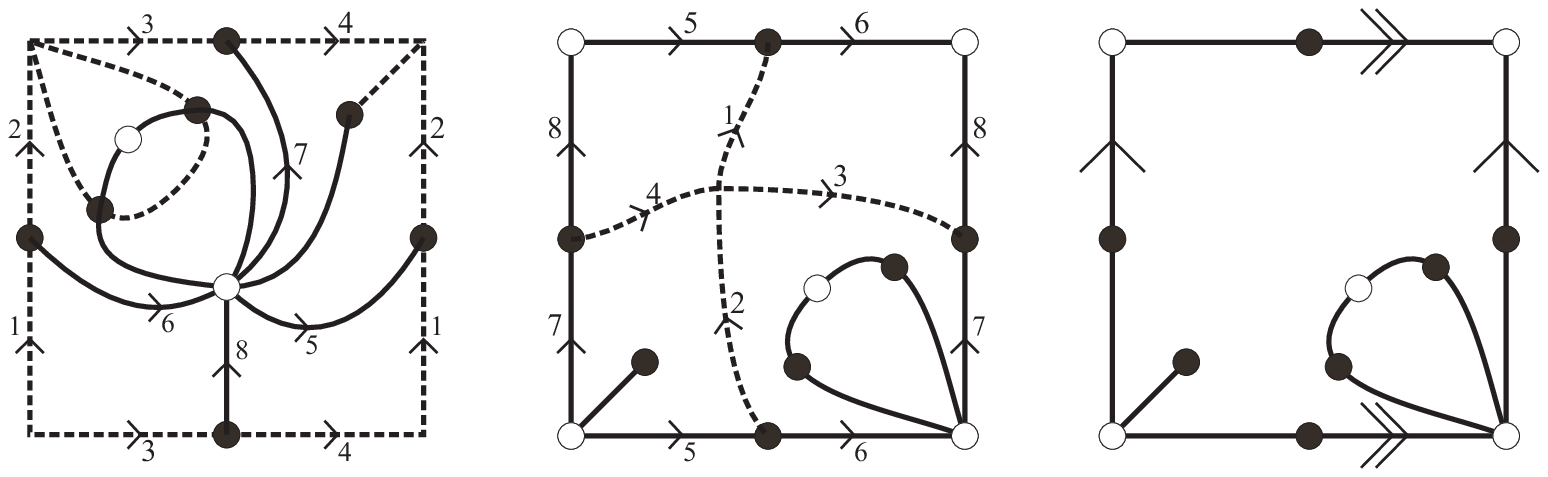}
    \end{center}
\mycap{$\secon(0,1,0,0)$ is self-dual. \label{genus1duality2:fig}}
\end{figure}

\begin{figure}
    \begin{center}
    \includegraphics[scale=.8]{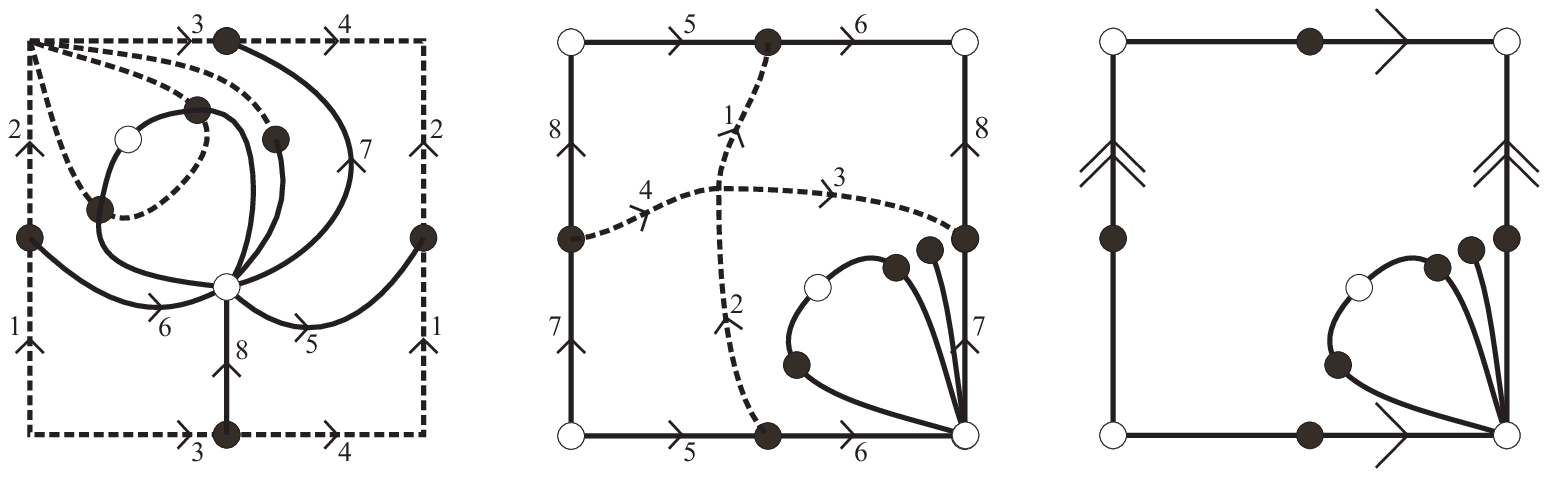}
    \end{center}
\mycap{$\terzo(0,1,0,0)$ is self-dual. \label{genus1duality3:fig}}
\end{figure}

\begin{figure}
    \begin{center}
    \includegraphics[scale=.8]{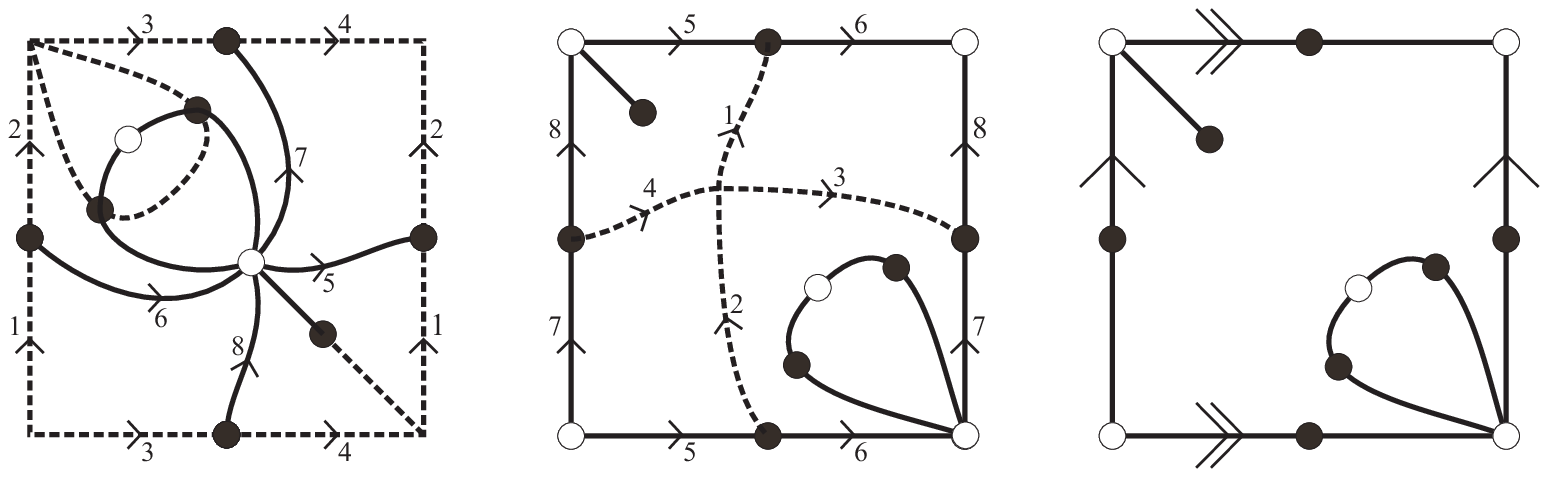}
    \end{center}
\mycap{$\quart(0,1,0,0)$ is self-dual. \label{genus1duality4:fig}}
\end{figure}

\begin{figure}
    \begin{center}
    \includegraphics[scale=.8]{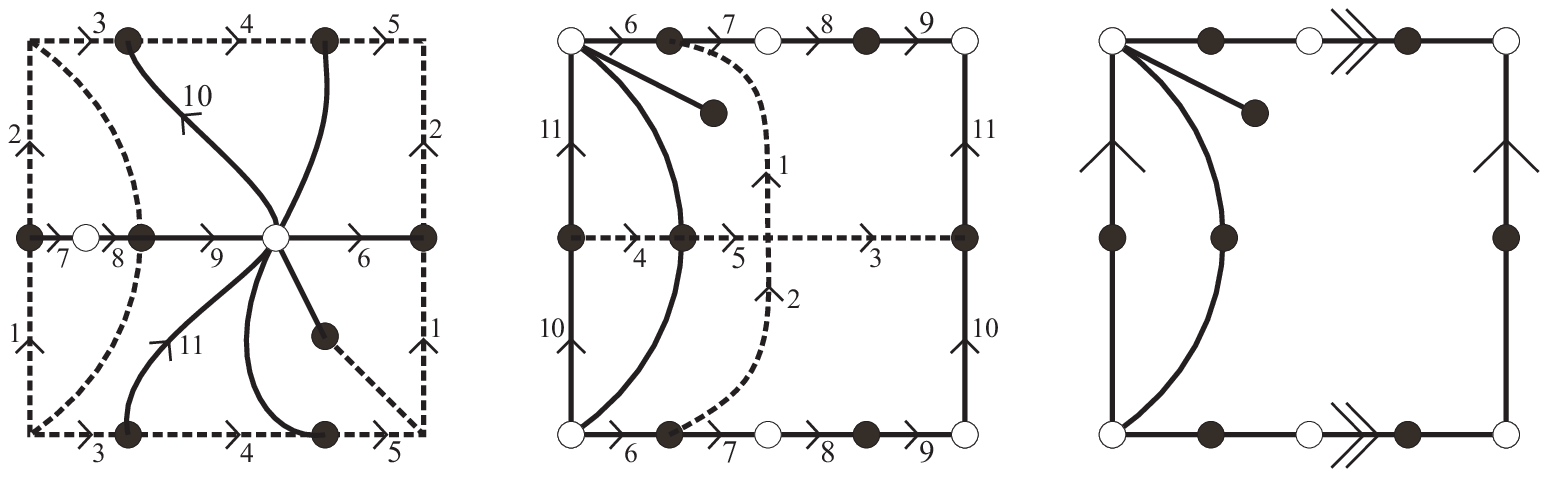}
    \end{center}
\mycap{$\sesto(0,1,0,0)$ is self-dual. \label{genus1duality6:fig}}
\end{figure}

\section{Genus 2}\label{genus2:sec}

The proof of Theorem~\ref{genus2:thm} employs that of
Theorem~0.1 in~\cite{x1}. In fact, it readily follows from~\cite{x1}
(see Fig.~12 there) that the embeddings in $2T$ of the graph $\Gamma$
of Fig.~\ref{abstractgraph:fig} are up to symmetry the $13$ ones shown
in Fig.~\ref{GammaintwoT:fig}.
\begin{figure}
    \begin{center}
    \includegraphics[scale=.6]{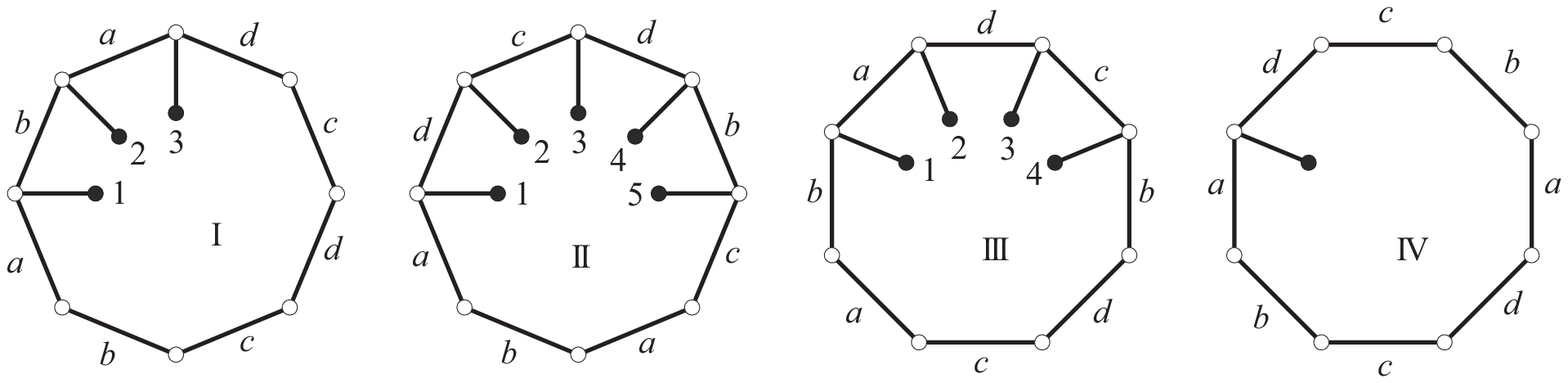}
    \end{center}
\mycap{Inequivalent embeddings of $\Gamma$ in $2T$ with a single disc as a complement. \label{GammaintwoT:fig}}
\end{figure}
Each of these four pictures shows an octagon whose edges should be paired
according to the labels, so that the octagon becomes $2T$ and its edges
become a bouquet $B$ of $4$ circles, which is part of the embedding of $\Gamma$
in $2T$. For each of the four embeddings \primo\ to \quart\ of $B$ in $\Gamma$,
the extra leg $\Gamma\setminus B$ of $\Gamma$ can be embedded in several
inequivalent ways, namely $3$ ways for \primo, 5 ways for \secon, 4 ways for \terzo, and only 1 for \quart, whence
the $13$ possibilities.
Let us denote by $e$ the label of the extra leg.
Note that there is a symmetry $(a,b,c,d,e)\leftrightarrow(b,a,d,c,e)$ in case $\primo.1$,
a (combinatorially equivalent) symmetry $(a,b,c,d,e)\leftrightarrow(d,c,b,a,e)$
in cases $\primo.3, \secon.2,\ \secon.5$ and $\quart$, and no other one.
It follows that the number $\nu(k)$ of realizations of $[2k+1]$ is $8$ times
\begin{itemize}
\item the number of ways of
expressing $k-4$ as $a+b+c+d+e$ with integer $a,b,c,d,e\geqslant 0$
\end{itemize}
plus $5$ times
\begin{itemize}
\item the number of ways of
expressing $k-4$ as $a+b+c+d+e$ with integer $a,b,c,d,e\geqslant 0$
up to the symmetry $(a,b,c,d,e)\leftrightarrow(b,a,d,c,e)$.
\end{itemize}
Replacing each of these integers with itself plus $1$, we see that
$\nu(k)$ is $8$ times
\begin{itemize}
\item the number of ways of
expressing $k+1$ as $a+b+c+d+e$ with integer $a,b,c,d,e\geqslant 1$,
\end{itemize}
namely $\binom{k}4$, plus $5$ times
\begin{itemize}
\item the number of ways of
expressing $k+1$ as $a+b+c+d+e$ with integer $a,b,c,d,e\geqslant 0$
up to the symmetry $(a,b,c,d,e)\leftrightarrow(b,a,d,c,e)$.
\end{itemize}
With this interpretation, in~\cite{x1} it was shown that
$\nu(k-1)$ is given by
$$\frac{k-1}{16}(7k^3-63k^2+197k-208)+\frac58(5-2k)\left[\frac k2\right].$$
Replacing $k$ by $k+1$ in this expression and noting that $\left[\frac{k+1}2\right]=k-\left[\frac k2\right]$
we get
$$\nu(k)=\frac k{16}(7k^3 - 42k^2 + 72k - 37)
+\frac58(2k-3)\left[\frac k2\right].$$
This value is correct except for the case $k=4$, where we have to take
into account the datum with repeated partitions in the last item
of Proposition~\ref{repetition:prop}, and we must analyze whether
any of these $13$ embeddings are dual to each other under
the last transformation generating
the equivalence $\sim$ of Theorem~\ref{equiv:Gamma:for:equiv:f:thm}.
This is done in Figg.~\ref{twoTdualityI:fig} to~\ref{twoTdualityIV:fig}, where
\begin{figure}
    \begin{center}
    \includegraphics[scale=.9]{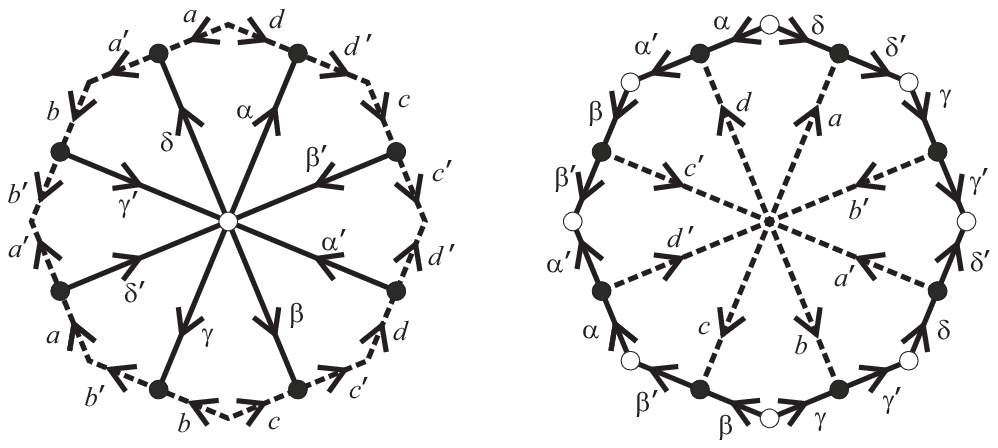}
    \end{center}
\mycap{Duals of the graphs of type $\primo.*$. On the left we show by dashed lines the edges
of the original graph giving a bouquet of 4 circles (so only the free leg of the graph is missing)
and by solid lines those of the dual graph (again, excluding the free leg).
The original edges are oriented and labelled as $a,a',\ldots,d,d'$ according to the way
they must be identified. The new edges are also oriented and labelled as $\alpha,\alpha',\ldots,\delta,\delta'$.
On the right we show the result of cutting along the $\alpha,\alpha',\ldots,\delta,\delta'$ and gluing along the
$a,a',\ldots,d,d'$. Since the new pattern of identifications is identical to the original one
(with Latin and Greek letters switched), we first of all
can conclude that the dual to any graph of type $\primo.*$ is also of type $\primo.*$.
More exactly, we see that the original extra leg of the graph $\primo.1$ is contained in
the quadrilateral with boundary $a'b'^{-1}\gamma'\delta'^{-1}$ on the left, hence the extra
leg of the dual is contained in the same quadrilateral on the right, which shows that the dual is also of
type $\primo.1$ (the position of the leg is not the same but it is combinatorially equivalent).
For $\primo.2$ the extra leg is in the quadrilateral $a'b\gamma'\delta$, so $\primo.2$ is self-dual
(different but equivalent position of the extra leg). Finally, for $\primo.3$ it is in
$a\delta^{-1}\alpha d^{-1}$, so also $\primo.3$ is self-dual (same position of the extra leg).\label{twoTdualityI:fig}}
\end{figure}
\begin{figure}
    \begin{center}
    \includegraphics[scale=.9]{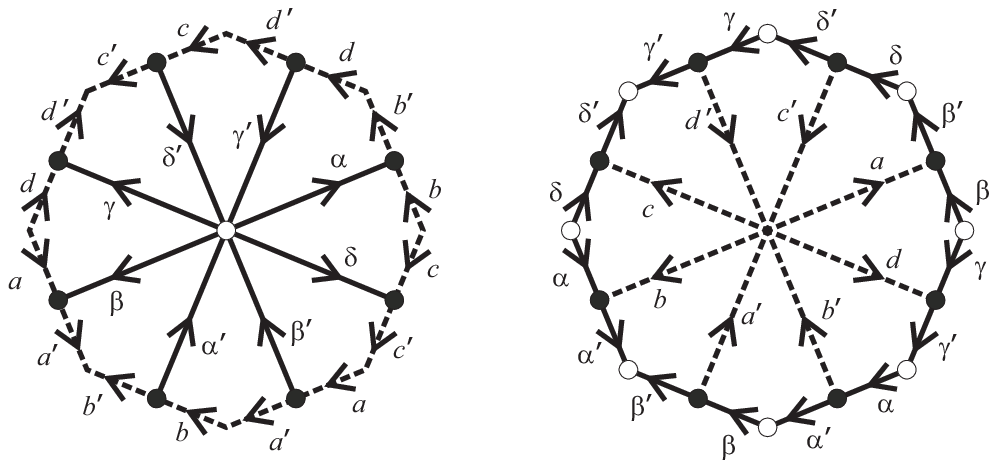}
    \end{center}
\mycap{Duals of the graphs of type $\secon.*$. These two images are explained as in the
previous figure and show that any graph of type $\secon.*$ is dual to another one of type $\secon.*$. 
More precisely, the original extra leg of $\secon.1$ lies in the quadrilateral
$a\beta^{-1}\gamma d^{-1}$, so the dual is $\secon.5$. For $\secon.2$ it lies
in $c'd'^{-1}\gamma^{-1}\delta'^{-1}$, so the dual is $\secon.3$, while for $\secon.4$
it lies in $b'd\gamma'\alpha$, so $\secon.4$ is self-dual (after duality the
new position of the leg is different but combinatorially equivalent).\label{twoTdualityII:fig}}
\end{figure}
\begin{figure}
    \begin{center}
    \includegraphics[scale=.9]{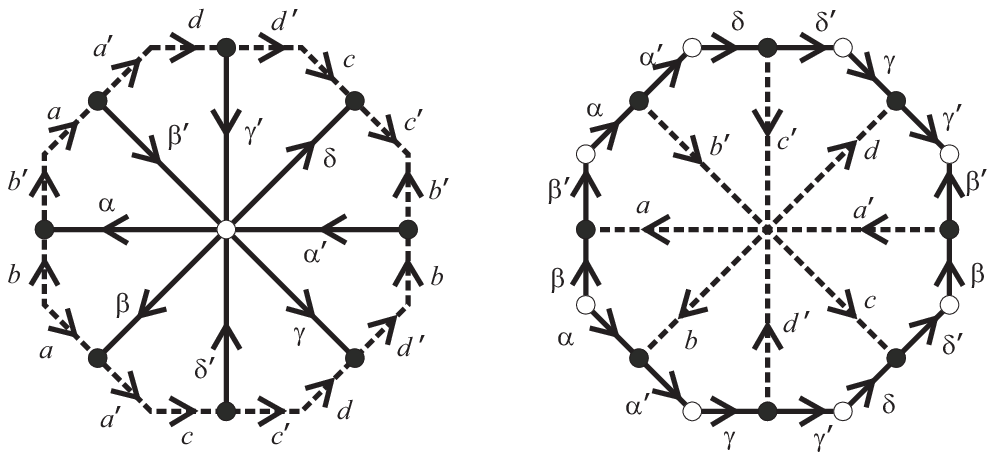}
    \end{center}
\mycap{Duals of the graphs of type $\terzo.*$. Again these pictures show as above 
that they can only be dual to each other. For $\terzo.1$ the extra leg is in 
$a\beta'\alpha b'$, so $\terzo.1$ is self-dual. For $\terzo.2$ it is
in $a'd\gamma'\beta'^{-1}$, so the dual is $\terzo.4$. Finally, for $\terzo.3$ it is in
$c\delta^{-1}\gamma'^{-1}d'$, so $\terzo.3$ is self-dual (different but equivalent 
position of the leg).\label{twoTdualityIII:fig}}
\end{figure}
\begin{figure}
    \begin{center}
    \includegraphics[scale=.9]{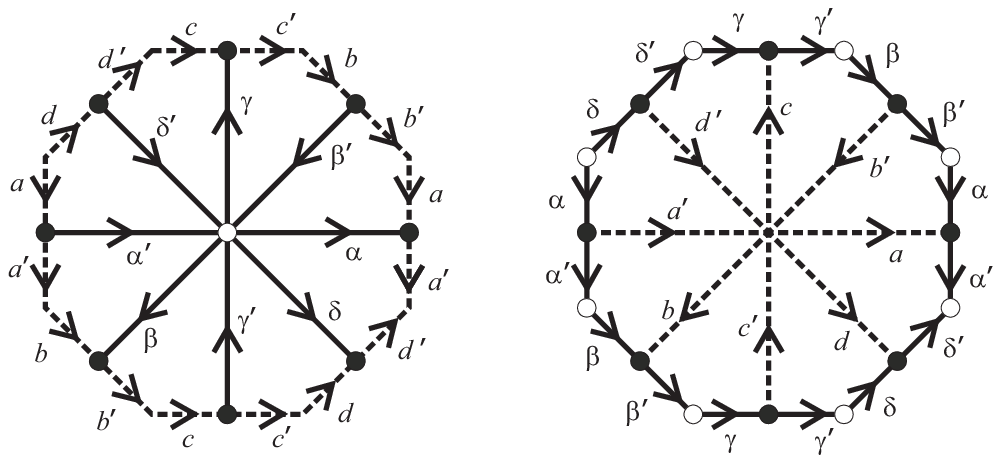}
    \end{center}
\mycap{The graph $\quart$ is self-dual (the position of the extra leg is immaterial). \label{twoTdualityIV:fig}}
\end{figure}
it is shown that each of the graphs
$$\primo.1\qquad\primo.2\qquad\primo.3\qquad\secon.4\qquad\terzo.1\qquad\terzo.3\qquad\quart$$
is self-dual, while
we have the following dualities:
$$\secon.1\leftrightarrow\secon.5\qquad
\secon.2\leftrightarrow\secon.3\qquad
\terzo.2\leftrightarrow\terzo.4$$
Therefore for $k=4$ we have $\nu=10$ rather than $\nu=13$.

\noindent
Dipartimento di Matematica\\
Universit\`a di Pisa\\
Largo Bruno Pontecorvo, 5\\
56127 PISA -- Italy\\
\texttt{petronio@dm.unipi.it}

\end{document}